\documentclass[11pt,reqno]{amsart}

\usepackage{setspace}
\usepackage{footnote}
\usepackage{graphicx}
\usepackage{cmap,mathtools}
\usepackage[T1]{fontenc}
\usepackage[utf8]{inputenc}
\usepackage[english]{babel}
\usepackage{amsfonts,amssymb,amsthm,amsmath,enumerate,bm,cite}
\usepackage{url}
\usepackage[shortlabels]{enumitem}
\usepackage{float}
\usepackage{secdot}

\allowdisplaybreaks
\usepackage{graphicx}
\usepackage{epstopdf}
\usepackage{a4wide}
\usepackage{xparse}
\usepackage[dvipsnames]{xcolor}
\usepackage[colorlinks=true,linktocpage,pdfpagelabels,bookmarksnumbered,bookmarksopen]{hyperref}

\usepackage{orcidlink}

\definecolor{deepblue}{rgb}{0.07, 0.04, 0.56}
\definecolor{mypink3}{cmyk}{0.7, 0.3, 0, 0.3}
\definecolor{myred}{rgb}{0.6, 0.05, 0.11}

\hypersetup{
	colorlinks=true,
	linkcolor=blue,         
	citecolor=blue,
	urlcolor=myred
}

\usepackage[margin=1.1in]{geometry}

\newtheorem{theorem}{Theorem}
\newtheorem{proposition}[theorem]{Proposition}
\newtheorem{lemma}[theorem]{Lemma}
\newtheorem{corollary}[theorem]{Corollary}
\newtheorem{definition}[theorem]{Definition}
\newtheorem{remark}[theorem]{Remark}

\let\OLDthebibliography\thebibliography
\renewcommand\thebibliography[1]{
	\OLDthebibliography{#1}
	\setlength{\parskip}{1pt}
	\setlength{\itemsep}{1pt plus 0.3ex}
}

\numberwithin{equation}{section} 
\numberwithin{theorem}{section}

\DeclarePairedDelimiter\norm{\lVert}{\rVert}%
\makeatletter
\let\oldnorm\norm
\def\norm{\@ifstar{\oldnorm}{\oldnorm*}}

\DeclareMathOperator*{\essinf}{ess\,inf}

\newcommand{\al} {\alpha}
\newcommand{\pa} {\partial}
\newcommand{\be} {\beta}
\newcommand{\de} {\delta}

\newcommand{\ga} {\gamma}

\newcommand{\Om} {\Omega}
\newcommand{\la} {\lambda}
\newcommand{\La} {\Lambda}
\newcommand{\si} {\sigma}

\newcommand{\noi} {\noindent}

\newcommand{\ra} {\rightarrow}

\newcommand{\intRn}{\displaystyle{\int_{\mathbb{R}^N}}}

\newcommand\restr[2]{{
  \left.\kern-\nulldelimiterspace 
  #1 
  \right|_{#2} 
  }}

\def\B{{\mathcal{B}}}

\def\E{{\mathcal E}}

\def\S{\mathbb{S}}

\def\R{{\mathbb R}}
\def\RN{{\mathbb R}^N}

\def\({{\Big(}}
\def\){{\Big)}}

\def\into{{\displaystyle \int_{\Omega}}}

\def\p{{p^{\prime}}}

\def\dx{{\rm d}x}
\def\dy{{\rm d}y}

\def\flap{{(-\Delta)^s}}
\def\uxy{{\frac{u(x)-u(y)}{|x-y|^{N+2s}}}}
 \def\uxys{{\frac{|u(x)-u(y)|^2}{|x-y|^{N+2s}}}}
\def\cns{{C_{N,s}}}

\def\Hsrn{{H^s(\mathbb{R}^N)}}

\def\uss{{[u]_{s}^2}}

\def\xsom{{H^s_0(\Omega)}}

\def\lmgv{{\la_1^{\text{min}}(g_0,V_0)}}
\def \ug {{\underline{g}}}
\def \uv {{\underline{V}}}

\def\H{{\mathcal{H}}}

\usepackage{mathtools}
\mathtoolsset{showonlyrefs} 

\usepackage[foot]{amsaddr}
\usepackage[pagewise]{lineno}

\makeatother

\date{}
\begin{document}
 \title[Optimization of the first weighted eigenvalue of the fractional Laplacian]{On the optimization of the first weighted eigenvalue of the fractional Laplacian}

\author[Mrityunjoy Ghosh]{Mrityunjoy Ghosh}

\address{Tata Institute of Fundamental Research,
Centre for Applicable
Mathematics\\
Post Bag No. 6503, Sharadanagar,
Bangalore 560065, India.}

\email{ghoshmrityunjoy22@gmail.com}



\maketitle
\begin{abstract}
In this article, we consider the minimization problem for the first eigenvalue of the fractional Laplacian with respect to the weight functions lying in the rearrangement classes of fixed weight functions. We prove the existence of minimizing weights in the rearrangement classes of weight functions satisfying some assumptions. Also, we provide characterizations of these minimizing weights in terms of the eigenfunctions. Furthermore, we establish various qualitative properties, such as Steiner symmetry, radial symmetry, foliated Schwarz symmetry, etc., of the minimizing weights and corresponding eigenfunctions.
\end{abstract}

\tableofcontents

\section{Introduction and statements of the main results}\label{intro}

Let $\Om\subset \R^N$ $(N\geq 2)$ be a smooth, bounded domain and $s\in (0,1).$ We consider the following weighted eigenvalue problem for the fractional Laplacian on $\Om$:
\begin{equation}\tag{E}\label{Problem}
	\left. \begin{aligned}
	\flap u +Vu &=\lambda gu \quad\text{in} \;\Omega,\\
	u &=0 \quad\; \;\;\;\text{in}\; \Om^c,
	\end{aligned}\right\}
\end{equation}
where $\la\in \R$, $V,g$ lie in some appropriate Lebesgue spaces, $\Om^c:=\RN\setminus\Om,$ and $\flap$ is the fractional Laplacian defined as below 
\begin{equation}\label{frac_lap}
    \flap u(x)=\cns\text{P.V.}\intRn\uxy \dy.
\end{equation}
Here, $\cns$ is a normalization constant and P.V. stands for the Cauchy principal value; cf. \cite{Valdinoci2012}. Let $\Hsrn, \xsom$ be the fractional Sobolev spaces (see Section \ref{Frac_sp} for more details). 
A real number $\la$ is called an eigenvalue of \eqref{Problem} if $\exists\; u\in \xsom\setminus\{0\}$ such that the following weak formulation holds:
\begin{align*}
    \frac{\cns}{2}\intRn\intRn \frac{(u(x)-u(y))(\phi(x)-\phi(y))}{|x-y|^{N+2s}}\dx\dy &+ \into u(x)\phi(x)V(x)\dx\\
    &=\la\into g(x)u(x)\phi(x)\dx,\;\forall\;\phi\in \xsom.
\end{align*}
For $u\in \Hsrn,$ let 
\begin{equation}\label{Rayleigh}
[u]_s:=\left(\intRn\intRn\uxys\dx\dy\right)^{\frac{1}{2}}.
\end{equation}
We define
\begin{equation}\label{variational}
    \la_1(g,V)=\inf\left\{ \frac{\frac{\cns}{2}\uss+\into V(x)|u(x)|^2\dx}{\into g(x)|u(x)|^2\dx} : u\in \xsom,\;\into g(x)|u(x)|^2\dx>0\right\}.
\end{equation}
Now, if $\la_1(g,V)$ is attained for some $u\in \xsom,$ then we call $\la_1(g,V)$ as the first eigenvalue and $u$ as the first eigenfunction of \eqref{Problem}.

 For $q>\frac{N}{2s}$ (observe that $N\geq 2s,$ for $s\in (0,1)$), let $C$ is the best constant of the fractional Sobolev embedding $\xsom \hookrightarrow L^{2q'}(\Om)$,
where $q'=\frac{q}{q-1}$ is the conjugate exponent of $q.$ We choose $g$ and $V$ satisfying the following:
\begin{equation}\tag{A}\label{assumption}
\left\{\begin{aligned}
    & \text{(A1)}\;\;\;g,V\in L^q(\Om),\;\text{where}\;q>\frac{N}{2s},\\
    & \text{(A2)}\;\;\;g^{+}\not\equiv 0,\;||V^{-}||_q\leq \frac{\frac{\cns}{2}-\de}{C^2},\;\text{for some}\;\de\in \left(0,\frac{\cns}{2}\right),
\end{aligned}\right.
\end{equation}
where $h^+=\max\{h,0\}$ and $h^-=\max\{-h,0\}$.

For $g=1,$ Del Pezzo, Bonder and Rios \cite[Theorem 3.3]{Bonder2018} proved the existence of the first eigenvalue $\la_1(g,V)$ of \eqref{Problem}. Furthermore, they showed that $\la_1(g,V)$ is a simple eigenvalue and the eigenfunctions associated to $\la_1(g,V)$ can be chosen to be positive; cf. \cite[Lemma 3.1]{Bonder2018}. The similar results have been obtained in \cite[Lemma 3.1]{LeeLee} for $V=0$ and nonnegative $g\in L^{\infty}(\Om)$ (see \cite{Cuesta2023} for sign-changing case). Indeed, for $g,V$ as in \eqref{assumption}, we can easily prove (see Theorem \ref{Eigen_exist} for an outline) that $\la_1(g,V)$ is achieved, i.e.,  $\la_1(g,V)$ is the first eigenvalue of \eqref{Problem}. Moreover, it is positive and simple. 

 Let $\E(\phi)$ denotes the rearrangement class of a measurable function $\phi:\Om\ra\R$, i.e., $\E(\phi)$ is the set of all measurable functions $\psi:\Om\ra\R$ satisfying 
 \begin{equation}\label{Rearrange_def}
 |\{x\in \Om: \psi(x)>t\}|=|\{x\in \Om: \phi(x)>t\}|,\;\forall\;t\in \R.
 \end{equation}
 Let $g_0,V_0$ satisfy \eqref{assumption}. In this article, we are interested in the following minimization problem:
\begin{equation}\tag{\text{Min}}\label{min_prob}
    \la_1^{\text{min}}(g_0,V_0):=\min\{\la_1(g,V): g\in \E(g_0), V\in \E(V_0)\}.
\end{equation}

In this article, we study the existence of minimizing weights for the minimization problem \eqref{min_prob}. Furthermore, we establish various symmetries of the minimizing weights and the corresponding eigenfunctions of \eqref{Problem} when the underlying domain has some symmetry. 

For the local case, optimization of the first eigenvalue of the Laplacian (or more generally for $p$-Laplacian and other differential operators) with respect to different types of weight functions has been extensively studied over the past few decades and it is commonly known as the \textit{composite membrane problem}. To quote a few, we refer to \cite{Chanillo, Ashbaugh, Cox, Bonder2006, Bonder2010, Radulescu2010, CuccuF, BonderDIE, Hamel, Kao, Berchio, Lamb}; see \cite[Chapter 9]{Henrot2006} also and the references therein for an overview of such problems. However, this provides a non-exhausting list of articles where similar problems are considered. 

\subsection{Existence of minimizing weights}

For $g_0=1$ and $V_0\in L^q(\Om),$ where $q\in (\frac{N}{2s},\infty),$ in \cite[Theorem 5.1]{Bonder2018}, the authors showed that there exists $\uv\in \E(V_0)$ such that $\lmgv=\la_1(g_0,\uv)$ (notice that, $\E(g_0)=\{g_0\}$ for $g_0=1$). In \cite{LeeLee}, Gonz\'{a}lez, Lee and Lee studied a particular case when $g_0=1$ and $V_0=\al\chi_D,$ for some $\al>0$ and  $D\subset\Om$. Along with several other qualitative properties, they proved the existence of a minimizing weight $\uv,$ which again takes the form $\uv=\al \chi_{\underline{D}}$, for some $\underline{D}\subset \Om$ having same Lebesgue measure as $D;$ cf. \cite[Theorem 1.1]{LeeLee}. By employing standard compactness arguments, we establish that there exists minimizing weights in \eqref{min_prob} even for more general weight functions $g_0,V_0$ satisfying \eqref{assumption}. This is the content of our first main result stated below.

\begin{theorem}\label{min_theo}
    Let $g_0,V_0$ be as given in \eqref{assumption} and $\normalfont{\lmgv}$ be as defined in \eqref{min_prob}. Then there exists $\ug\in \E(g_0)$ and $\uv\in \E(V_0)$ such that $\normalfont{\lmgv}$ is achieved at $\ug,\uv,$ i.e.,
    $$\normalfont{\lmgv}=\la_1(\ug,\uv).$$
\end{theorem}

\subsection{Symmetry of minimizing weights} 
Let $\Om\subset \R^N$ be Steiner symmetric with respect to $\{x_1=0\}$, $g_0=1$ and $V_0=\al\chi_D,$ where $\al>0$ and $D\subset \Om.$ For such $g_0$ and $V_0,$ in \cite[Theorem 1.2]{LeeLee}, the authors showed that the minimizing weights (which have essentially the similar form as $V_0$) are symmetric with respect to $\{x_1=0\}$ and the corresponding first eigenfunctions are Steiner symmetric with respect to $\{x_1=0\}$. In the below theorem, we prove that the minimizing weights in \eqref{min_prob} and the associated first eigenfunctions of \eqref{Problem} inherit similar properties even for more general weight functions $g_0$.

\begin{theorem}[Steiner symmetry]\label{steiner_theo}
     Let $V_0=0$ and $\ug$ be a minimizing weight given by Theorem \ref{min_theo}. Let $u$ be a first eigenfunction of \eqref{Problem} associated to $\ug.$ Assume that $\Om\subset \R^N$ is Steiner symmetric with respect to the hyperplane $\pa H,$ where $H$ is an open affine half-space in $\R^N$ such that $0\in \overline{H}$. Then $u,\ug$ are Steiner symmetric in $\Om$ with respect to $\pa H$.
\end{theorem}

Following similar ideas as in the local case \cite{Chanillo, Anedda}, to prove \cite[Theorem 1.2]{LeeLee}, the authors used the Steiner symmetrization along with approximating the kernel for proving the equality case.  However, as observed recently by Biswas, Das and the author \cite[Section 1.2]{UNM} in the local case, the Steiner symmetrization and the characterization of the equality case of P\'{o}lya-Szeg\"{o} type inequality are not helpful to prove the Steiner symmetry when the weight functions are not of the form $\chi_D$ (for some $D\subset \Om$). We stress that similar difficulties occur in the non-local case, too. Now, adapting the approach of \cite{UNM} to the non-local setting, we overcome this difficulty by proving the invariant structure of the minimizing weights and the corresponding first eigenfunctions under $polarization$ (see Section \ref{pola_sec} for a precise definition). More explicitly, we first show that if the domain $\Om$ is invariant with respect to polarization, then the same is propagated into the minimizing weights and the corresponding eigenfunctions also (Proposition \ref{pola_propo}). Next, we use a characterization of the Steiner symmetry in terms of the polarization to conclude the result. 

Now, as a direct consequence of the Steiner symmetry (Theorem \ref{steiner_theo}), we prove the radiality of the minimizing weights and the associated first eigenfunctions in a ball.
\begin{corollary}[Radial symmetry]\label{rad_theo}
Let $\Om=B_r\subset \R^N$ be an open ball of radius $r>0.$ Further, assume that $V_0$, $\ug$ and $u$ be as in Theorem \ref{steiner_theo}. Then $u$ and $\ug$ are radial in $\Om.$
\end{corollary}

So far, we have encountered situations where the minimizing weights and the associated eigenfunctions inherit the symmetry of the underlying domain. However, it need not be the case always. To describe the symmetry-breaking phenomena, we consider the annular domain as follows: for $0<r<R,$ 
\begin{equation}\label{annulus}
    \B_t=B_R(0)\setminus \overline{B_r(te_1)},\;\text{for}\;0\leq t<R-r,
\end{equation} 
where $B_s(x)$ denotes the ball of radius $s$ centered at $x$ and $e_1=(1,0,\dots,0)\in \R^N.$ In the local case, symmetry-breaking was first observed by Chanillo, Grieser, Imai, Kurata, and Ohnishi in \cite[Theorem 6]{Chanillo}. For $g_0=1$ and $V_0=\al\chi_D$ $(D\subset \Om),$ they showed that there exists a thin annular domain $\B_0$ (for sufficiently large $r$) such that the minimizing weights are not radial in $\B_0.$ As a consequence of this asymmetry, it is easy to observe that the minimizing weight is not unique. A similar phenomenon holds in the non-local case, too, as shown by Gonz\'{a}lez, Lee and Lee in \cite[Theorem 1.4]{LeeLee} (see Remark \ref{asym} also). Motivated by these asymmetry results, it is, therefore, natural to look for some ``weaker'' symmetry of the minimizing weights and the associated eigenfunctions on such annular domains. In this direction, to the best of our knowledge, the first result appeared recently in \cite[Theorem 1.5]{UNM} for the local case, where the authors proved that the minimizing weights and the associated first eigenfunctions are foliated Schwarz symmetric (see Definition \ref{Foliated_symmetry} for precise definition) in $\B_t.$ In the following, we establish analogous results for the minimizers in \eqref{min_prob} and the associated first eigenfunctions of \eqref{Problem}.

\begin{theorem}[Foliated Schwarz symmetry]\label{foliated_theo}
    Suppose $\B_t$ is as defined by \eqref{annulus} and $g_0,V_0\geq 0$ are as mentioned in Theorem \ref{min_theo}. Let $\ug,\uv$ be the minimizing weights given by Theorem \ref{min_theo} and $u$ be an associated first eigenfunction of \eqref{Problem} on $\B_t.$ Then the following statements hold:
    \begin{enumerate}[(i)]
        \item if $t=0,$ then $u,\ug$ are foliated Schwarz symmetric in $\B_t$ with respect to some vector $\be\in \S^{N-1}$, while $\uv$ is foliated Schwarz symmetric in $\B_t$ with respect to $-\be.$

        \item if $t>0$ and $V_0=0,$ then $u,\ug$ are foliated Schwarz symmetric in $\B_t$ with respect to $-e_1$.
    \end{enumerate}
\end{theorem}

\begin{remark}
It is necessary to highlight that the symmetry obtained in the above theorem is independent of the size of the annular domain $\B_t$ defined in \eqref{annulus}. Indeed, this is in contrast to the cases where radial symmetry fails as described earlier; cf. \cite[Theorem 6]{Chanillo} for the local case and \cite[Theorem 1.4]{LeeLee} for the non-local case. However, as a consequence of the previous theorem, we show that (Corollary \ref{compo}) the axial symmetry is still preserved by the minimizers in the typical composite membrane problem considered in \cite{LeeLee}.
\end{remark}

To prove Theorem \ref{foliated_theo}, again, we exploit the polarization invariant structure of the minimizing weight functions and the eigenfunctions of \eqref{Problem}. The first step is to observe that the annular domain $\B_t$, for $t\geq 0,$ is symmetric about the $x_1$-axis. Moreover, $\B_0$ remains unchanged under polarization with respect to any half-space containing the origin in its boundary. Similarly, $\B_t,$ for $t>0,$ remains unchanged under polarization with respect to any half-space containing the origin in its boundary and containing $-e_1.$ Then, we achieve the conclusion by utilizing the aforementioned properties of $\B_t$ and using a representation of the foliated Schwarz symmetry in terms of polarization.  

The rest of the article is arranged in the following way. In the next section, we collect some preliminary results, which will be used later. In particular, fractional Sobolev spaces, the existence of the first eigenvalue and several symmetrizations are discussed in Section \ref{Preli}. Proofs of the main results (Theorem \ref{min_theo}, Theorem \ref{steiner_theo} and Theorem \ref{foliated_theo}) are presented in Section \ref{main_results}.

\section{Preliminaries}\label{Preli}
In this section, we introduce some useful tools that will be needed to prove our main results. First, we recall the definitions of fractional Sobolev spaces that we are going to work with and state some basic properties of these spaces. Then, we prove the existence results regarding the first eigenvalue of \eqref{Problem}. In the last two subsections, we state various properties of the polarizations and define several symmetries, respectively. 

\subsection{Fractional Sobolev spaces}\label{Frac_sp} 
For $s\in (0,1),$ let
\begin{equation*}
[u]_s :=\left(\intRn\intRn\uxys\dx\dy\right)^{\frac{1}{2}}
\end{equation*}
be the Gagliardo semi-norm of $u$ in $\R^N.$
The fractional Sobolev space $\Hsrn$ (also 
 called as Slobodeckij, Aronszajn or Gagliardo space)  is defined as below
\begin{equation}\label{Sobolev_sp}
\begin{aligned}
    \Hsrn =\left\{u \in L^2(\R^N): [u]_s^2 <\infty\right\}.
    \end{aligned}
\end{equation}
For $u,v\in \Hsrn,$ let $\langle u,v\rangle_{\Hsrn}$ be given by 
\begin{equation*}\label{inner_pr}
    \langle u,v\rangle_{\Hsrn}:= \intRn u(x)v(x)\dx + \intRn\intRn \frac{(u(x)-u(y))(v(x)-v(y))}{|x-y|^{N+2s}}\dx\dy.
\end{equation*}
We can easily verify that $\langle \cdot,\cdot\rangle_{\Hsrn}$ defines an inner product on $\Hsrn$. Let $||\cdot||_{\Hsrn}$ be the norm induced by the aforementioned inner product, i.e., for $u\in \Hsrn,$ we have
\begin{equation*}
    ||u||_{\Hsrn}=\left(||u||^2_{L^2(\R^N)}+[u]_s^2\right)^\frac{1}{2}.
\end{equation*}
Then $\Hsrn$ is a Hilbert space with respect to the norm $||\cdot||_{\Hsrn}$; see, for instance, \cite[Section 1.2.2]{Radulescu} or \cite[Lemma 6 and Lemma 7]{Servadei2012}. Let $\Om\subset \R^N$ be a bounded domain. We define the Sobolev space $\xsom$, a subspace of $\Hsrn,$ as below
\begin{equation*}
    \xsom =\{u \in \Hsrn: u\equiv 0 \;\text{in }\;\Om^c\}.
\end{equation*}
 Observe that, for $u\in \xsom,$ we have 
\begin{equation*}\label{hs_norm}
||u||_{\xsom}:=||u||_{\Hsrn}=\left(||u||^2_{L^2(\Om)}+[u]_s^2\right)^\frac{1}{2}.
\end{equation*}
Now, we recall some properties of the 
Sobolev space $\xsom.$

\begin{theorem} \label{Sobolev_theo}
Let $\Om\subset \R^N$ be a smooth, bounded domain and $s\in (0,1).$ Then the following holds:
    \begin{enumerate}[(i)]
        \item The space $\xsom$ is a Hilbert space with respect to the norm $||\cdot||_{\xsom}.$

        \item Let $2_s^*:=\frac{2N}{N-2s}$ be the fractional critical Sobolev exponent of $2.$ Then the embedding $\xsom\hookrightarrow L^r(\Om)$ is compact for all $r\in[1,2_s^*).$

        \item The embedding $\xsom\hookrightarrow L^{2_s^*}(\Om)$ is continuous.
    \end{enumerate}
\end{theorem}
\begin{proof}
    $(i)$ Proof follows by using \cite[Lemma 1.28 and Lemma 1.29]{Radulescu}.

    $(ii)$ and $(iii)$ follows from \cite[Lemma 1.31]{Radulescu} (see also \cite[Lemma 9]{Servadei2015}).
\end{proof}

\subsection{Existence of the first eigenvalue}\label{eigen_sec}
In this subsection, we show the existence of the first eigenvalue of \eqref{Problem} and state some properties of the first eigenfunctions. The proof follows by adapting a similar method, as described in the proof of \cite[Theorem A.3, Appendix]{UNM}, to the non-local case; see \cite[Theorem 3.3]{Bonder2018} also for the case $g=1$. We only give a sketch here.

Let $\la_1(g,V)$ be as defined by \eqref{variational}, i.e., 
$$\la_1(g,V)=\inf\left\{ \frac{\frac{\cns}{2}\uss+\into V(x)|u(x)|^2\dx}{\into g(x)|u(x)|^2\dx} : u\in \xsom,\;\into g(x)|u(x)|^2\dx>0\right\}.$$

\begin{theorem}\label{Eigen_exist}
    Let $\Om\subset \R^N$ be a smooth, bounded domain and $g,V$ satisfy the assumptions of \eqref{assumption}. Then we have the following:
    \begin{enumerate}[(i)]
        \item There exists $u\in \xsom$ such that $\la_1(g,V)$ is attained at $u$ and $\la_1(g,V)>0.$ In particular, $\la_1(g,V)$ is the first eigenvalue and $u$ is a first eigenfunciton of \eqref{Problem}.

        \item There exists first eigenfunction $w\in \xsom$ such that $w$ can be chosen to be positive (a.e.) in $\Om,$ i.e.,  $w>0$ a.e. in $\Om.$

        \item Furthermore, $\la_1(g,V)$ is simple, i.e., if $\la_1(g,V)$ is attained at both $u,w\in \xsom,$ then there exists $c\in \R$ such that $u=cw$ a.e. in $\Om.$ 
    \end{enumerate}
\end{theorem}
\begin{proof}
    $(i)$ Observe that, using the homogeneity, $\la_1(g,V)$ can be
    equivalently characterized as 
    \begin{equation*}
        \la_1(g,V)=\inf\left\{ \frac{\cns}{2}\uss+\into V(x)|u(x)|^2\dx: u\in \xsom,\;\into g(x)|u(x)|^2\dx=1\right\}.
    \end{equation*}
    Let $(\phi_n)\in \xsom$ be a minimizing sequence for $\la_1(g,V).$ Then we can show that (see, for instance, the proof of Lemma \ref{un_bound}) $(\phi_n)$ is actually bounded in $\xsom.$ Therefore, using the standard compactness arguments, we get a limit $u\in \xsom$ of the sequence $\phi_n$ such that $\la_1(g,V)$ is attained at $u.$ Consequently, we have $\la_1(g,V)>0.$

    $(ii)$ Let $u$ be a first eigenfunction associated to $\la_1(g,V)$. Now using the fact $||u(x)|-u(y)|\leq |u(x)-u(y)|,$ we have $[|u|]_s^2\leq [u]_s^2.$ Thus $|u|$ is also a first eigenfunction associated to $\la_1(g,V)$. Moreover, using \cite[Theorem 2.9]{Bonder2018}, we get $|u|>0$ a.e. in $\Om.$ Now the conclusion follows by choosing $w=|u|.$

    $(iii)$ The proof follows by similar arguments as in \cite[Theorem 3.6]{Bonder2018} using the Picone-type identity \cite[Lemma 6.2]{Amghibech}.
    
\end{proof}

\subsection{Polarization}\label{pola_sec}
Let $\H$ be the collection of all open affine half-spaces in $\R^N.$ For $H\in \H,$ let $\sigma_H$ denotes the reflection with respect to the hyperplane $\pa H.$ Now, we define the polarization of a measurable set and a measurable function; cf. \cite[Section 5]{BrockSolynin} or \cite[Definition 1.2]{UNM}.

\begin{definition}\label{pola_def}
    Let $\Om\subset \R^N$ be a measurable set and $\phi:\R^N\ra\R$ be a measurable function. Suppose $H\in \H.$
    \begin{enumerate}[(i)]
        \item The polarization $\Om_H$ of $\Om$ with respect to $H$ is defined as 
        $$\Om_H=((\Om\cup \sigma_H(\Om))\cap H)\cup (\Om\cap \sigma_H(\Om)).$$

        \item The polarization $\phi_H$ of $\phi$ with respect to $H$ is defined by 
        \[
        \phi_H(x)= \left\{
        \begin{array}{@{}l@{\thinspace}l}
        \max \{\phi(x),\phi(\sigma_H(x))\}, \;\;& \text{if}\; x \in H,\\
       \min \{\phi(x),\phi(\sigma_H(x))\}, \;\;& \text{if}\  x \in H^c.
     \end{array}
        \right.
        \] 
         If $\Om\subsetneq \R^N$ and $\phi:\Om\ra\R$ is a measurable function, then the polarization $\phi_H$ of $\phi$ is given by $\phi_H:=\tilde{\phi}_H|_\Om,$
        where $\tilde{\phi}$ is the zero extension of $\phi$ to $\R^N$ as below
        \[
        \tilde{\phi}(x)= \left\{
        \begin{array}{@{}l@{\thinspace}l}
        \phi(x), \;\;& \text{if}\; x \in \Om,\\
       0, \;\;& \text{if}\  x \in \Om^c.
     \end{array}
        \right.
        \] 

        \item The dual-polarization $\phi^H$ of $\phi$ on $\si_H(\Om)$ is defined as $\phi^H=\phi_H\circ \si_H.$
    \end{enumerate}
\end{definition}

The following lemma states a few immediate facts about the polarization of a domain; for a proof, see \cite[Proposition 2.1]{UNM}.
\begin{lemma}\label{Pola_domain}
    Let $H\in \H$ and $\Om\subset \R^N$ $(N\geq 2)$ be a domain such that $\Om_H=\Om.$ Then we have the following:
    \begin{enumerate}[(i)]
        \item $\si_H(\Om^c\cap H)\subset \Om^c\cap H^c$ and $\si_H(\Om\cap \overline{H}^c)\subset \Om\cap H$.

        \item if $\si_H(\Om)\neq \Om,$ then there exists $A\subset \Om\cap H$ with $|A|>0$ such that $\si_H(A)\subset \Om^c\cap \overline{H}^c.$
    \end{enumerate}
\end{lemma}


Now, we prove some properties related to the fractional Sobolev norm of a function and its polarization.
\begin{proposition}\label{frac_norm}
    Let $H\in \H,$ $\Om\subset \R^N$ be a domain satisfying $\Om_H=\Om$ and $\phi:\Om\rightarrow\R$ be a nonnegative measurable function. Also, let $\phi_H$ be the polarization of $\phi$ with respect to $H.$ Then the following holds:
    \begin{enumerate}[(i)]
        \item $\phi_H$ is a rearrangement of $\phi.$
        
        \item In addition, let $\phi \in \Hsrn$. Then we have $$[\phi_H]^2_s\leq [\phi]^2_s.$$
    
    \item Let $\phi$ be as in $(ii).$ Then $$||\phi_H||_{\Hsrn}\leq ||\phi||_{\Hsrn}.$$
    In particular, we get $\phi_H\in \Hsrn.$

    \item If $\phi\in \xsom,$ then $\phi_H\in \xsom.$
    \end{enumerate}
    
\end{proposition}
\begin{proof}
    $(i)$ For a proof, we refer to \cite[Proposition 2.2]{UNM} or \cite[Section 5]{BrockSolynin}.

    $(ii)$ Proof follows using \cite[Lemma 8.1]{BrockSolynin}.

 $(iii)$ By $(i),$ we have $\phi_H$ is a rearrangement of $\phi.$ Hence $||\phi||_{L^2(\R^N)}=||\phi_H||_{L^2(\R^N)}$ (cf. \cite[Proposition 2.3]{Schaftingen}). Thus the statement follows from $(ii).$

 $(iv)$ Let $\phi \in\xsom.$ Then $\phi \in \Hsrn$ and $\phi =0$ in $\Om^c.$ Let $x\in \Om^c.$ We consider two cases. If $x\in H,$ then by Lemma \ref{Pola_domain}-$(i)$, we get $\si_H(x)\in  \Om^c.$ Thus $\phi (\si_H(x))=0$ and hence using Definition \ref{pola_def}-$(ii)$, we have $\phi_H(x)=0.$ If $x\in  H^c,$ then again using Definition \ref{pola_def}-$(ii)$, we deduce that $\phi_H(x)=0$ (since $\phi\geq 0$). Therefore, $\phi_H=0$ in $\Om^c.$ Moreover, $(iii)$ yields $\phi_H\in \Hsrn.$  Hence we get $\phi_H\in \xsom.$
\end{proof}

Throughout the article, we do not distinguish two functions $f$ and $g$ defined on a domain $\Om\subset \R^N$ if they are equal a.e. in $\Om$, i.e., 
$$\text{`}f=g\; \text{in}\;\Om\text{'}\;\text{essentially means that}\; \text{`}f=g\;\text{a.e. in}\;\Om.\text{'}$$

\begin{remark}\label{Some_remarks} 
Here, we state some facts about the polarization that is immediate from the definitions and the properties as mentioned above. Let $\Om\subset \R^N$ be measurable, $H\in \H$ and $\phi:\Om\rightarrow\R$ be a nonnegative measurable function.
    \begin{enumerate}[(i)]
        \item Suppose $\phi\geq\phi\circ \si_H$ in $\Om\cap H.$ Then we have $\phi\leq \phi\circ \si_H$ in $\Om\cap \overline{H}^c$. Hence from Definition \ref{pola_def}-$(ii),$ we must have $$\phi_H=\phi\;\text{in}\;\Om.$$
        Similarly, if $\phi\leq \phi\circ \si_H$ in $\Om\cap H,$ it follows that $\phi^H=\phi$ in $\Om.$

        \item Observe that, $\overline{H}^c\in \H.$ Now, if $\Om=\Om_{\overline{H}^c},$ then Proposition \ref{pola_propo}-$(i)$ infers that $\phi_{\overline{H}^c}$ is a rearrangement of $\phi.$ Moreover, using Definition \ref{pola_def}-$(iii),$ we get $\phi^H=\phi_{\overline{H}^c}.$ Therefore, if $\Om=\Om_{\overline{H}^c},$ then $\phi^H$ is a rearrangement of $\phi.$ 
        
        \item If $\Om$ is symmetric with respect to $\pa H,$ then we have $\Om_H=\Om_{\overline{H}^c}=\Om$. Hence in this case, using $(ii)$ and Proposition \ref{pola_propo}-$(i),$ respectively, it follows that both $\phi^H$ and $\phi_H$ are rearrangements of $\phi.$ 
 
    \end{enumerate}
\end{remark}

Next, we state Hardy-Littlewood and reverse Hardy-Littlewood inequality for polarization.
\begin{lemma}\label{Hardy}
    Let $p\in [1,\infty),$ $H\in \H$ such that $0\in \overline{ H},$ and $\Om\subset \R^N$ be a bounded domain with $\Om_H=\Om.$ Suppose $u,v\in L^p(\Om)$ are such that $uv
    \in L^1(\Om).$ Then the following holds:
    \begin{enumerate}[(i)]
        \item (Hardy-Littlewood) Let either $u$ or $v$ is nonnegative. Then 
        $$\int_{\Om} u(x)v(x)\dx \leq \int_{\Om} u_H(x)v_H(x)\dx.$$

        \item (Reverse Hardy-Littlewood) Let $v$ is nonnegative. Then 
        $$\int_{\Om} u^H(x)v_H(x)\dx \leq \int_{\Om} u(x)v(x)\dx.$$
    \end{enumerate}
\end{lemma}
\begin{proof}
    For a proof, we refer to \cite[Proposition 2.5]{UNM}.
\end{proof}

\subsection{Symmetrizations and symmetries}
This subsection is devoted to recalling the definitions of various symmetrizations and the symmetries of a measurable function, along with their characterizations.

\subsubsection{Steiner symmetry}
We start with the definition of Steiner symmetrization; cf. \cite[Section 2.2]{Henrot2006}. First, we introduce the following notations: 
\begin{itemize}
    \item $x:=(x',x_N)\in \R^N$, where $x'=(x_1,x_2,\dots,x_{N-1})\in \R^{N-1}$ and $x_N\in \R$,

    \item $\pi_{N-1}$ is the orthogonal projection from $\R^N$ to $\R^{N-1}$,

    \item for measurable $\Om\subset \R^N$, $\Om_{x'}:=\{x_N\in \R:(x',x_N)\in\Om\}$ is the slice of $\Om$ through $x'$ in the direction $x_N$,
    \item $\{x_N=0\}$ is the hyperplane $\{x=(x',x_N)\in \R^N: x_N=0\},$
    \item for $1\leq m\leq N,$ $|\cdot|_m$ denotes the $m$-dimensional Lebesgue measure.
\end{itemize}

\begin{definition}[Steiner symmetric domain]\label{Steiner_dom}
The \textit{Steiner symmetrization} $\Om^\#$ of $\Om$ with respect to the hyperplane $\{x_N=0\}$ is defined as $$\Om^\#=\left\{x=(x',x_N)\in \R^N: |x_N|< \frac{|\Om_{x'}|_1}{2}, x'\in \pi_{N-1}(\Om)\right\}.$$
Then $\Om$ is said to be \textit{Steiner symmetric} with respect to the hyperplane $\{x_N=0\}$  if $\Om=\Om^\#$ (up to translation). 
\end{definition}

\begin{remark}\label{RemS}
    Equivalently, $\Om$ is Steiner symmetric with respect to the hyperplane $\{x_N=0\}$ if 
    \begin{enumerate}[(i)]
        \item $\Om$ is symmetric with respect to the hyperplane $\{x_N=0\}$, and

        \item $\Om$ is convex with respect to the $x_N$-axis, i.e., any line segment parallel to the $x_N$-axis joining two points in $\Om$ lies completely inside $\Om$.
    \end{enumerate}
\end{remark}
 
\begin{definition}[Steiner symmetric function]\label{Steiner_func}
Let $\Om\subset \mathbb{R}^N$ be a measurable set and $f:\Om\rightarrow\R$ be a nonnegative measurable function. Then the \textit{Steiner symmetrization} $f^\#$ of $f$ on $\Om^\#$ with respect to the the hyperplane $\{x_N=0\}$ is defined by
$$f^\#(x)=\sup\left\{c\in \R:x\in \{y\in \Om:f(y)\geq c\}^\#\right\}, \;\text{for}\;x\in \Om^\#.$$
Let $\Om=\Om^\#$. If $f=f^\#$ in $\Om$, then $f$ is called \textit{Steiner symmetric} with respect to the hyperplane $\{x_N=0\}$.
\end{definition}
 Next, we give a characterization of Steiner symmetric domains and Steiner symmetric functions using polarization; cf. \cite[Lemma 6.3]{BrockSolynin}.
 \begin{proposition}\label{Steiner_char}
 Let $\Om$ and $f$ be as in Definition \ref{Steiner_func}. Also, let $\H_*\subset \H$ be the collection of all half-spaces $H$ such that $H$ contains the hyperplane
 $\{x_N=0\}$ and $\pa H$ is parallel to the hyperplane $\{x_N=0\}$. Then the following holds:
 \begin{enumerate}[(i)]
     \item $\Om=\Om^\#$ if and  only if $\Om=\Om_H,$ for all $H\in \H_*$.
     \item if $\Om=\Om^\#$, then $f$ is Steiner symmetric with respect to the hyperplane $\{x_N=0\}$ if and  only if $f=f_H,$ for all $H\in \H_*$.
 \end{enumerate}
 \end{proposition}

\subsubsection{Foliated Schwarz symmetry}\label{fol_sub}
In this subsection, we define the foliated Schwarz symmetry of a measurable function. First, we describe the following collections of half-spaces in $\R^N$:
\begin{align*}
    \H_0 &=\{H\in \H: 0\in \overline{H}\},\;\;\;\quad\widehat{\H}_0=\{H\in \H_0: 0\in \pa H\}, \\
    \H(\ga) &=\{H\in \H:\ga\in H\}, \; \, \widehat{\H}_0(\ga)=\{H\in\widehat{\H}_0 :\ga\in H\},\;\text{for}\;\ga\in \R^N.
\end{align*}

Next, we define the foliated Schwarz symmetrization of a function on radial domains; cf. \cite[Definition 3.3]{Brock2020}.

\begin{definition}[Foliated Schwarz symmetrization]\label{Foliated_sz} Let $\Om\subset\R^N$ be a radial domain with respect to 0 and  $f:\Om\rightarrow \R$ be a nonnegative measurable function. Then the \textit{foliated Schwarz symmetrization} $f^*$ of $f$ with respect to a vector $\ga\in \S^{N-1}$ is the function satisfying the following properties:
\begin{enumerate}[(i)]
     \item $f^*(x)=h(r,\theta)$, $\forall\; x\in \Om$, for some function $h:[0,\infty)\times [0,\pi)\rightarrow\R$, which is decreasing in $\theta$, where $(r,\theta):=\big(|x|,{\rm{arccos}}(\frac{x\cdot \ga}{|x|})\big)$.
     \item for $a,b\in\mathbb{R}$ with $a<b$ and $r\geq 0$, $$|\{x:|x|=r,\; a<f(x)\leq b\}|_{N-1}=|\{x:|x|=r,\; a<f^*(x)\leq b\}|_{N-1},$$
     where $|\cdot|_{N-1}$ denotes the $(N-1)$-dimensional Lebesgue measure.
\end{enumerate}
\end{definition}

Now, we state the definitions of foliated Schwarz symmetry of a function defined on a ball or annular domains. Recall that, $\B_t=B_R(0)\setminus \overline{B_r(te_1)},$ for $0\leq t<R-r.$

\begin{definition}[Foliated Schwarz symmetric function]\label{Foliated_symmetry}
Let $\B_t$ be as given above.
\begin{enumerate}[(i)]
    \item Let $\Om\subset \R^N$ be a radial domain (either ball or $\B_0$) with respect to 0. Then a nonnegative measurable function $f:\Om\rightarrow \R$ is said to be \textit{foliated Schwarz symmetric} in $\Om$ with respect to a vector $\ga\in \S^{N-1}$ if $f=f^*$.

    \item For $0<t< R-r$, let $f:\B_t\rightarrow\R$ be a nonnegative measurable function. Define $\tilde{f}:B_R(0)\rightarrow \R$ as follows
    \[
        \tilde{f}(x)= \left\{
        \begin{array}{@{}l@{\thinspace}l}
        f(x), \;\;& \text{if}\; x \in \B_t,\\
       0, \;\;& \text{if}\  x \in \overline{B_r(te_1)}.
     \end{array}
        \right.
        \]
    Then $f$ is called \textit{foliated Schwarz symmetric} in $\B_t$ with respect to $-e_1$ if $\tilde{f}$ is foliated Schwarz symmetric in $B_R(0)$ with respect to $-e_1$.
\end{enumerate}
\end{definition}

From the above definitions, it follows that if $f$ is foliated Schwarz symmetric with respect to $\ga\in \S^{N-1}$, then $f$ is axially symmetric with respect to the axis $\R \ga$ and decreasing in the polar angle $\theta={\rm{arccos}}\big(\frac{x\cdot\ga}{|x|}\big)$. Now, we state a characterization of foliated Schwarz symmetry in terms of polarization. 

\begin{proposition} \label{Foliated_char}
Let $\Om\subset\R^N$ be a radial domain with respect to 0 and $f\in L^p(\Om)$ $(1\leq p<\infty)$ be nonnegative. 
\begin{enumerate}[(i)]
    \item If for every $H\in \widehat{\H}_0$, either $f_H=f$ or $f^H=f$ in $\Om$, then there exists $\ga\in \S^{N-1}$ such that $f$ is foliated Schwarz symmetric in $\Om$ with respect to $\ga$.
    
    \item Let $\beta\in \S^{N-1}$. Then $f$ is foliated Schwarz symmetric in $\Om$ with respect to $\beta$ if and only if $f_H=f$ in $\Om,$ for all $H\in \widehat{\H}_0(\beta)$.
\end{enumerate}
\end{proposition}
\begin{proof}
    $(i)$ For a proof, we refer to \cite[Theorem 3.5]{Brock2020}.
    
    $(ii)$ The conclusion follows from \cite[Proposition 2.4]{Weth2010} for continuous functions. However, we can easily obtain the same result for measurable functions by using a similar approach as described in \cite[Theorem 3.5]{Brock2020}. We skip the detailed proof here.
\end{proof}

\begin{corollary}\label{Foliated_noncon}
Let $f:\B_t\rightarrow\R$ be a nonnegative measurable function. Then $f$ is foliated Schwarz symmetric in $\B_t$ with respect to $-e_1$ if and only if $f_H=f\;\text{in}\;\B_t,$ for all $H\in \widehat{\H}_0(-e_1).$
\end{corollary}

\begin{proof}
    Applying Definition \ref{Foliated_symmetry}-$(ii)$ and Proposition \ref{Foliated_char}-$(ii)$,  we get that $f:\B_t\rightarrow\R$ is foliated Schwarz symmetric in $\B_t$ with respect to $-e_1$ if and only if $\tilde{f}_H=\tilde{f}\;\text{in}\;B_R(0),\;\forall\,H\in \widehat{\H}_0(-e_1)$. 
    Observe that, by Definition \ref{Foliated_symmetry}-$(ii)$, we always have $\tilde{f}_H=\tilde{f}$ in $\overline{B_r(te_1)},$ for all $H\in\widehat{\H}_0(-e_1)$. As a consequence, $f$ is foliated Schwarz symmetric in $\B_t$ with respect to $-e_1$ if and only if $\tilde{f}_H=\tilde{f}\;\text{in}\;B_R(0)\setminus\overline{B_r(te_1)}=\B_t$, i.e., $$f_H=f\;\text{in}\;\B_t,\;\text{for all}\; H\in \widehat{\H}_0(-e_1).$$
\end{proof}

\section{Proofs of the main results}\label{main_results}
We establish our main results in this section. First, we prove the existence of the minimizing weights in \eqref{min_prob} (Theorem \ref{min_theo}) and provide some characterizations of these weights in terms of the associated first eigenfunctions of \eqref{Problem}. Next, we prove various symmetry results (Theorem \ref{steiner_theo}, Corollary \ref{rad_theo} and Theorem \ref{foliated_theo}) concerning the minimizing weights and the eigenfunctions.

\subsection{Existence and characterization} 
We start this subsection with some useful consequences of rearrangements that are required to establish the main results. 

Recall that, $\E(f)$ denotes the collection of all rearrangements of a measurable functions $f:\Om\rightarrow\R.$ For $f\in L^p(\Om),$ where $p\in[1,\infty),$ let $\overline{\E(f)}$ be the weak closure of $\E(f)$ in $L^p(\Om).$

\begin{lemma}\label{rearrange_p}
Let $f_0 \in L^p(\Om),\;\text{where}\;p \in [1, \infty)$. Then the following statements hold:
\begin{enumerate}[(i)]
    \item If $f \in \E(f_0)$, then $\norm{f}_p = \norm{f_0}_p$ and $f^{\pm} \in \E(f_0^{\pm})$.
    
    \item Let $h \in L^{\p}(\Om)$. Then there exists $\phi,\psi \in \E(f_0)$ such that 
\begin{align*}
    \into \phi(x) h(x) \, \dx \le \into f(x) h(x) \, \dx \le \into \psi(x) h(x) \, \dx,\;\; \forall \, f \in  \overline{\E(f_0)}.
\end{align*}

\end{enumerate}
\end{lemma}

\begin{proof}
 $(i)$ The first part of the statement follows from \cite[Lemma 2.1]{Burton}. The second part is a straightforward consequence of the definition of rearrangement.

 $(ii)$ For a proof, we refer to  \cite[Lemma 2.4]{Burton}.
\end{proof}

In order to prove the existence results, we need the following lemma. Recall that, the fractional semi-norm $[\cdot]_s$ is given by 
$$[u]_s=\left(\intRn\intRn\uxys\dx\dy\right)^{\frac{1}{2}},\;\text{for}\;u\in \Hsrn.$$

\begin{lemma}\label{un_bound}
Let $u\in \xsom$ and $V_0$ satisfies the assumptions of \eqref{assumption}. Then there exists $\al>0$ such that
\begin{equation*}
    \frac{\cns}{2}\uss+\into V(x)|u(x)|^2\dx\geq \al \uss,\;\text{for all} \;V\in \E(V_0).
\end{equation*}
\end{lemma}
\begin{proof}
      By the hypothesis, $V_0\in L^q(\Om),$ where $q>\frac{N}{2s}.$ Let $V\in \E(V_0).$ Then using H\"{o}lder inequality and the embedding $\xsom\hookrightarrow L^{2q'}(\Om)$, we get
      \begin{equation}\label{v-}
          \into V^-(x)|u(x)|^2\dx\leq ||V^-||_q||u||^2_{2q'}\leq C^2 ||V^-||_q \uss,\;\text{for}\;u\in \xsom.
      \end{equation}
      Since $V$ is a rearrangement of $V_0,$ by Lemma \ref{rearrange_p}-$(i)$, we have $V^-\in \E(V^-_0)$ and hence $||V^-||_r=||V_0^-||_r,$ for $r\geq 1.$ In particular, using the assumption, it follows that
      $$||V^-||_q\leq \frac{\frac{\cns}{2}-\de}{C^2},\;\text{for some}\;\de\in \left(0,\frac{\cns}{2}\right).$$ 
      Therefore, applying \eqref{v-}, we obtain
      \begin{align*}
          \frac{\cns}{2}\uss+\into V(x)|u(x)|^2\dx &\geq \frac{\cns}{2}\uss-\into V^-(x)|u(x)|^2\dx\\
          &\geq \left\{\frac{\cns}{2}-C^2 ||V^-||_q\right\} \uss\\
          &\geq \de \uss,\;\text{for}\;u\in \xsom.
      \end{align*}
      Hence the conclusion follows by taking $\al=\de.$
\end{proof}

Now, we are in a stage to prove Theorem \ref{min_theo}.
\begin{proof}[Proof of Theorem \ref{min_theo}]
    Let $g_0,V_0$ be as stated in \eqref{assumption}, i.e., 
\begin{equation*}
\left\{\begin{aligned}
    & \;g_0,V_0\in L^q(\Om),\;\text{where}\;q>\frac{N}{2s},\\
    & \;g_0^{+}\not\equiv 0,\;||V_0^{-}||_q\leq \frac{\frac{\cns}{2}-\de}{C^2},\;\text{for some}\;\de\in \left(0,\frac{\cns}{2}\right).
\end{aligned}\right.
\end{equation*}
By the definition \eqref{min_prob} of $\la_1^{\text{min}}(g_0,V_0),$    we have
    \begin{equation}
        \la_1^{\text{min}}(g_0,V_0)=\min\{\la_1(g,V): g\in \E(g_0), V\in \E(V_0)\}.
    \end{equation}
    For the notational simplicity, we write $\La_1:=\la_1^{\text{min}}(g_0,V_0).$ Let $(g_n,V_n)\in \E(g_0)\times \E(V_0)$ be a minimizing sequence for $\La_1,$ i.e.,
    \begin{equation}\label{lamin}
        \La_1=\displaystyle\lim_{n\rightarrow \infty} \la_1(g_n,V_n).
    \end{equation}
    Observe that, $g_n$ and $V_n$ satisfy all the assumptions as $g_0$ and $V_0,$ respectively (using Lemma \ref{rearrange_p}-$(i)$). Therefore, Theorem \ref{Eigen_exist} implies that there exists a sequence $(u_n)\in \xsom$ consisting of positive first eigenfunctions associated to $(g_n,V_n)$ and it can be chosen in such a way that $\into g_n(x)(u_n(x))^2\dx=1,$ for all $n\geq 1.$ Hence we have
    \begin{equation}\label{gnun}
    \la_1(g_n,V_n)=\frac{\cns}{2}[u_n]_s^2+\into V_n(x)(u_n(x))^2\dx,\;\text{for all}\;n\geq 1.
    \end{equation}
Now, Lemma \ref{un_bound} infers that the sequence $(u_n)$ is bounded in $\xsom.$ Therefore, using the reflexivity of $\xsom$ (since it is Hilbert space by Theorem \ref{Sobolev_theo}-$(i)$), we get a subsequence $(u_{n_k})$ of $(u_n)$ and $u\in \xsom$ such that 
$$u_{n_k}\rightharpoonup u\;\text{in}\;\xsom.$$
Since $q>\frac{N}{2s},$ we get $2q'<2_s^*,$ where $2_s^*=\frac{2N}{N-2s}$. Therefore, we apply the compactness (Theorem \ref{Sobolev_theo}-$(ii)$) of the embedding $\xsom\hookrightarrow L^{2q'}(\Om)$ to obtain 
\begin{equation}\label{unk}
u_{n_k}\rightarrow u\;\text{in}\;L^{2q'}(\Om).
\end{equation}
Moreover, we have $u> 0$ a.e. in $\Om.$ On the other hand, since $(g_{n_k},V_{n_k})\in \E(g_0)\times \E(V_0),$ the following holds by Lemma \ref{rearrange_p}-$(i)$: $$||g_{n_k}||_q=||g_0||_q\;\text{and\;}||V_{n_k}||_q=||V_0||_q,\;\text{for all}\;k\geq 1.$$
Thus using the reflexivity of $L^{2q'}(\Om),$ up to a subsequence, we get
\begin{equation}\label{weak}
    g_{n_k}\rightharpoonup \hat{g}\;\text{and}\; V_{n_k}\rightharpoonup \hat{V}\;\text{in}\;L^{2q'}(\Om).
\end{equation}
Observe that, $(\hat{g},\hat{V})\in \overline{\E(g_0)}\times \overline{\E(V_0)}$, where $\overline{\E(g_0)}$ and $\overline{\E(V_0)}$ are weak closures of $\E(g_0)$ and $\E(V_0),$ respectively, in $L^{2q'}(\Om).$ Now, using \eqref{unk} and \eqref{weak}, we arrive at 
\begin{equation}\label{inteq}
\begin{aligned}
    \displaystyle\lim_{k\rightarrow \infty} \into g_{n_k}(x) (u_{n_k}(x))^2\dx &=\into \hat{g}(x) (u(x))^2\dx=1\\
    \text{and}\;\lim_{k\rightarrow \infty} \into V_{n_k}(x) (u_{n_k}(x))^2\dx &=\into \hat{V}(x) (u(x))^2\dx.
\end{aligned}
\end{equation}
Applying both \eqref{inteq} and \eqref{gnun} in \eqref{lamin}, and using the lower semi-continuity of the fractional semi-norm $[\cdot]_s$, we obtain
\begin{align}\label{lowersemi}
\La_1 &=\displaystyle\lim_{k\rightarrow \infty}\left\{\frac{\cns}{2}[u_{n_k}]_s^2+\into V_{n_k}(x)(u_{n_k}(x))^2\dx\right\}\nonumber\\
& \geq \frac{\cns}{2}[u]_s^2+\into \hat{V}(x)(u(x))^2\dx.
\end{align}
Now, we use Lemma \ref{rearrange_p}-$(ii)$ to get  $(\ug,\uv)\in \E(g_0)\times \E(V_0)$ satisfying the following:
\begin{equation}\label{burt}
    \into \hat{g}(x) (u(x))^2\dx\leq \into \ug(x) (u(x))^2\dx,\;\;\into \hat{V}(x)(u(x))^2\dx\geq \into \uv(x)(u(x))^2\dx.
\end{equation}
Substituting the above estimates in \eqref{lowersemi} and using the definition \eqref{min_prob} of $\La_1$, we deduce
\begin{align}
    \La_1\geq \frac{\cns}{2}[u]_s^2+\into \hat{V}(x)(u(x))^2\dx\geq \frac{\frac{\cns}{2}[u]_s^2+\into \uv(x)(u(x))^2\dx}{\into \ug(x) (u(x))^2\dx}\geq \La_1.
\end{align}
Hence $\La_1=\la_1^{\text{min}}(g_0,V_0)=\la_1(\ug,\uv).$ This completes the proof.
\end{proof}

In the following, we provide a characterization of the minimizers in \eqref{min_prob}, given by Theorem \ref{min_theo}, in terms of the associated first eigenfunctions of \eqref{Problem}.

\begin{proposition}\label{Repre}
    Let $\ug\in \E(g_0)$ and $\uv \in \E(V_0)$ be the minimizers of $\normalfont{\lmgv}$ as given by Theorem \ref{min_theo}. Suppose $u$ is a positive first eigenfunction of \eqref{Problem} associated to $\ug$ and $\uv.$ Then the following statements hold:
    \begin{enumerate}[(i)]
        \item There exists an increasing function $\phi:\R\rightarrow \R$ such that $$\ug =\phi \circ u\;\;a.e.\;\text{in}\;\Om.$$

        \item There exists a decreasing function $\psi:\R\rightarrow \R$ such that $$\uv =\psi \circ u\;\;a.e.\;\text{in}\;\Om.$$
    \end{enumerate}
\end{proposition}
\begin{proof}
    $(i)$ By the hypothesis, we have 
    \begin{equation}\label{hyp}
    \normalfont{\lmgv}=\la_1(\ug,\uv).
    \end{equation}
    Let $J:\E(g_0)\rightarrow \R$ be defined as below
    $$J(g)=\into g(x)|u(x)|^2\dx,\;\text{for}\;g\in \E(g_0).$$
    First, we show that $\ug$ is a maximizer of $J$ in $\E(g_0).$ If possible, let there exists $g\in \E(g_0)$ such that $$c:=\into g(x)|u(x)|^2\dx>\into \ug(x)|u(x)|^2\dx.$$
    Let $w=\frac{u}{c^{\frac{1}{2}}}.$ Then we have $\into g(x)|w(x)|^2\dx=1$ and
    \begin{equation*}
        \begin{aligned}
    \la_1(g,\uv)\leq\frac{\frac{\cns}{2}[w]_s^2+\into \uv(x)|w(x)|^2\dx}{\into g(x)|w(x)|^2\dx} &=\frac{1}{c}\left\{\frac{\cns}{2}[u]_s^2+\into \uv(x)|u(x)|^2\dx\right\}\\
    &< \frac{\frac{\cns}{2}[u]_s^2+\into \uv(x)|u(x)|^2\dx}{\into \ug(x)|u(x)|^2\dx}=\la_1(\ug,\uv).
        \end{aligned}
    \end{equation*}
    Therefore, using \eqref{hyp} and the definition \eqref{min_prob} of $\normalfont{\lmgv}$, we get  $$\normalfont{\lmgv}=\la_1(\ug,\uv)\leq \la_1(g,\uv) < \la_1(\ug,\uv)=\normalfont{\lmgv},$$ a contradiction. Thus $\ug$ is a maximizer of $J$ in $\E(g_0).$ Now, we prove that $\ug$ is the unique maximizer of $J$ in $\E(g_0).$ To achieve that, let there exists another maximizer $h\in \E(g_0)$ of $J.$ Then $\into h(x)|u(x)|^2\dx=\into \ug(x)|u(x)|^2\dx.$ Let $\tilde{u}$ be an eigenfunction of \eqref{Problem} associated to $\la_1(h,\uv).$ Now,
    \begin{equation*}
        \begin{aligned}
            \la_1(\ug,\uv)\leq \la_1(h,\uv) =\frac{\frac{\cns}{2}[\tilde{u}]_s^2+\into \uv(x)|\tilde{u}(x)|^2\dx}{\into h(x)|\tilde{u}(x)|^2\dx}
            &\leq \frac{\frac{\cns}{2}[u]_s^2+\into \uv(x)|u(x)|^2\dx}{\into h(x)|u(x)|^2\dx}\\
            &= \frac{\frac{\cns}{2}[u]_s^2+\into \uv(x)|u(x)|^2\dx}{\into \ug(x)|u(x)|^2\dx}=\la_1(\ug,\uv).
        \end{aligned}
    \end{equation*}
    Thus we have the equality in the above inequalities and consequently, $u$ is also a first eigenfunction of \eqref{Problem} associated to $\la_1(h,\uv).$ Therefore, the following equations are satisfied weakly: 
    \begin{align*}
        \flap u +\uv u &=\la_1(\ug,\uv) \ug u \quad\text{in} \;\Omega,\\
        \flap u +\uv u &=\la_1(\ug,\uv) h u \quad\text{in} \;\Omega.
    \end{align*}
    Subtracting the above equations, we get $\la_1(\ug,\uv) (\ug-h) u=0\;\text{a.e. in}\;\Omega.$ Since $u$ is a positive first eigenfunction, we conclude that $h=\ug$ a.e. in $\Om.$ Thus $\ug$ is the unique maximizer of $J$ in $\E(g_0).$ Hence, using \cite[Theorem 5]{Burton1987}, we get $\phi:\R\rightarrow \R$ such that $\ug =\phi \circ u\;\text{a.e. in}\;\Om.$

    $(ii)$ Define $\Tilde{J}:\E(V_0)\rightarrow \R$ as follows
    $$\Tilde{J}(V)=\into V(x)|u(x)|^2\dx,\;\text{for}\;V\in \E(V_0).$$
    Then, following similar arguments as in $(i),$ we can show that $\uv$ is the unique minimizer of $J$ in $\E(V_0).$ Therefore, again using \cite[Theorem 5]{Burton1987}, we get a decreasing function $\psi:\R\rightarrow \R$ such that $\uv =\psi \circ u\;\text{a.e. in}\;\Om.$
\end{proof}

\begin{remark}
    For the local case, a similar characterization of the minimizers as given by Proposition \ref{Repre} can be found in \cite[Theorem 3.5]{Bonder2010}.
\end{remark}

\subsection{Symmetries} 
In this subsection, we prove various symmetry results of the minimizers and the associated first eigenfunctions. First, we recall a version of maximum principles for the fractional Laplacian.

\begin{proposition}[Strong maximum principle]\label{Strong_max2}
    Let $H\subset \R^N$ be a half-space, $\Om\subset H$ be a bounded domain and $U\in L^\infty(\Om)$. Let $v\in \xsom$ satisfies (weakly)
    \begin{align*}
       \flap v &\geq U v \;\;\text{in}\;\;\Om.
   \end{align*}
    Suppose that $v \geq 0$ in $H$ and $v$ is antisymmetric with respect to the hyperplane $\pa H.$ Then 
    $$\text{either}\;v\equiv 0 \;\text{in}\;\R^N,\;\text{or}\;\essinf\limits_K\;v>0,\;\text{for all compact set}\;K\subset \Om.$$
\end{proposition}

\begin{proof}
    Proof follows from \cite[Proposition 3.6]{Jarohs2016}.
\end{proof}

Now, we prove some invariance structure of the minimizers and the associated first eigenfunctions under the polarization. To be precise, we prove that the minimizers remain unchanged under the polarization if the domain is polarization invariant. This invariance property of the minimizers will be crucial to prove various symmetry results. Recall that, $\H$ is the collection of all open affine half-spaces in $\R^N.$

\begin{proposition}\label{pola_propo}
    Let $H\in \H$ with $0\in \overline{H}$ and $\Om\subset \R^N$ be a smooth, bounded domain such that $\Om_H=\Om.$ Assume that, $g_0,V_0$ are as in Theorem \ref{min_theo} and both are non-negative. Let $\ug,\uv$ be the minimizing weights as in Theorem \ref{min_theo} and $u$ be a positive first eigenfunction of \eqref{Problem} associated to $\ug$ and $\uv.$ Then the following holds:
    \begin{enumerate}[(i)]
        \item if $\si_H(\Om)\neq \Om$ and $V_0=0,$ then $u_H=u$ and $\ug_H=\ug$ in $\Om.$
        \item if $\si_H(\Om)= \Om,$ then either $u_H=u,\ug_H=\ug,\uv^H=\uv$ in $\Om,$ or $u^H=u,\ug^H=\ug,\uv_H=\uv$ in $\Om.$ 
    \end{enumerate}
\end{proposition}
\begin{proof}
    $(i)$ For simplicity, we write $\la_1^{\text{min}}(g_0):= \la_1^{\text{min}}(g_0,0)$ and $ \la_1(\ug):=\la_1(\ug,0).$  Let $u$ be a positive first eigenfunction corresponding to $\ug.$ Now, by assumption, we have 
    \begin{equation}\label{vari1}
        \la_1^{\text{min}}(g_0)=\la_1(\ug)=\frac{\frac{\cns}{2}\uss}{\into \ug(x)|u(x)|^2\dx}.
    \end{equation}
    Let $H\in \H$ be such that $0\in \overline{H}.$ 
    We observe that $\ug_H\in \E(g_0)$ (Proposition \ref{frac_norm}-$(i)$) and $\ug_H$ satisfy (by Lemma \ref{rearrange_p}-$(i)$) all the assumptions of Theorem \ref{Eigen_exist}. Therefore, $\la_1(\ug_H)$ is attained. Indeed, we show that $\la_1(\ug_H)$ is attained at $u_H.$
    Now, using Proposition \ref{frac_norm}-$(ii)$ and the Hardy-Littlewood inequality (Lemma \ref{Hardy}-$(i)$), respectively, we get
    \begin{align}\label{hardy1}
        [u_H]_s^2\leq [u]_s^2\;;\;\into \ug_H(x)|u_H(x)|^2\dx\geq \into \ug(x)|u(x)|^2\dx.
    \end{align}
     Since $u>0$, by Proposition \ref{frac_norm}-$(iv),$ we have $u_H\in \xsom.$ Therefore, using \eqref{hardy1} in \eqref{vari1}, we obtain
    \begin{align}
        \la_1^{\text{min}}(g_0)=\frac{\frac{\cns}{2}\uss}{\into \ug(x)|u(x)|^2\dx} &\geq \frac{\frac{\cns}{2}[u_H]_s^2}{\into \ug_H(x)|u_H(x)|^2\dx}\\
        &\geq \la_1(\ug_H)\geq \la_1^{\text{min}}(g_0).
    \end{align}
    Thus we have the equality in the above inequalities. Therefore, $\la_1^{\text{min}}(g_0)=\la_1(\ug_H)$ and $u_H$ is a first eigenfunction associated to $\la_1(\ug_H).$ Hence the following equations hold weakly in $\Om$:
    \begin{align}
        \flap u_H  &=\la_1^{\text{min}}(g_0) \ug_H u_H \;\;\text{in} \;\Omega,\label{eqh}\\
        \flap u &=\la_1^{\text{min}}(g_0) \ug u \quad\quad\text{in} \;\Omega.\label{eqq}
    \end{align}
    Define $w(x)=u_H(x)-u(x),$ for $x\in \R^N.$ Then using Definition \ref{pola_def}-$(ii)$, we get $w(x)\geq 0,$ for $x\in \Om\cap H.$ Moreover, if $x\in \Om^c\cap H,$ Lemma \ref{Pola_domain}-$(i)$ yields that $\si_H(x)\in \Om^c\cap H^c$. Therefore, $u(\si_H(x))=u(x)=0,$ for $x\in \Om^c\cap H.$ Hence $w(x)=0,$ for $x\in \Om^c\cap H.$ Thus $$w\geq 0\; \text{in} \;H.$$ Furthermore, using the definition of polarization (Definition \ref{pola_def}), it is easy to observe that, $w$ is antisymmetric with respect to the hyperplane $\pa H.$ Now, from \eqref{eqh} and \eqref{eqq}, $w$ satisfies the below equation weakly
    \begin{equation}\label{weq}
        \flap w  =\la_1^{\text{min}}(g_0) (\ug_H u_H-\ug u) \;\;\text{in} \;\Omega\cap H.
    \end{equation}
     Observe that, $\ug_H\geq 0$ in $\Om$ as $\ug\geq 0$ in $\Om$. Therefore, $\ug_H u_H-\ug u\geq 0$ in $\Om\cap H.$ Hence from \eqref{weq}, we get
     \begin{equation}\label{weqq}
        \flap w\geq 0\;\;\text{in} \;\Omega\cap H.
    \end{equation}
     Since both $u,u_H\in \xsom$, it follows that $w\in \xsom.$ Thus $w$ is antisymmetric, $w\geq 0$ in $H$ and $w$ satisfies \eqref{weqq}. Therefore, applying the Strong maximum principle (Proposition \ref{Strong_max2}) in $\Om\cap H$, we get either $w>0$ in $\Om \cap H$ or $w\equiv 0$ in $\R^N.$ Equivalently, $$\text{either}\;u_H>u\;\text{in}\; \Om\cap H,\; \text{or}\; u_H\equiv u\;\text{in}\; \R^N.$$ 
     Since $\si_H(\Om)\neq \Om,$ using Lemma \ref{Pola_domain}-$(ii),$ we get $A\subset \Om\cap H$ such that $|A|>0$ and $\si_H(A)\subset \Om^c\cap \overline{H}^c.$ Therefore, $u(\si_H(x))=0,$ for all $x\in A.$ Moreover, $u(x)>0,$ for $x\in A$. Hence $u_H(x)=u(x),$ for all $x\in A.$ Thus we must have $u_H\equiv u$ in $\R^N$ and hence $$u_H= u\;\text{in}\; \Om.$$ 
     Substituting $u_H=u$ in \eqref{eqh}, \eqref{eqq} and subtracting, we conclude $$\ug_H=\ug\;\text{in}\;\Om.$$

     $(ii)$ By the assumption, we have 
     \begin{equation}\label{lmv}
         \la_1^{\text{min}}(g_0,V_0)=\la_1(\ug,\uv)=\frac{\frac{\cns}{2}\uss+\into \uv(x)|u(x)|^2\dx}{\into \ug(x)|u(x)|^2\dx}.
     \end{equation}
     Let $H\in \H$ be as given in the statement and $\si_H(\Om)=\Om.$ Thus $\Om$ is symmetric with respect to $\pa H$ and hence using Remark \ref{Some_remarks}-$(iii)$, we get $\ug_H\in \E(g_0)$ and  $\uv^H\in \E(V_0).$ Besides, applying Theorem \ref{Eigen_exist}, it follows that $\la_1(\ug_H,\uv^H)$ is achieved. We claim that $\la_1(\ug_H,\uv^H)$ is achieved at $u_H.$ Now, from the Hardy-Littlewood and reverse Hardy-Littlewood inequality (Lemma \ref{Hardy}-$(i)$ and $(ii)$), we have 
     \begin{equation}\label{Hard}
         \into \ug_H(x)|u_H(x)|^2\dx\geq \into \ug(x)|u(x)|^2\dx;\;\;\into \uv^H(x)|u_H(x)|^2\dx \leq \into \uv(x)|u(x)|^2\dx.
     \end{equation}
     Furthermore, Proposition \ref{frac_norm}-$(ii)$ yields $[u_H]_s^2\leq [u]_s^2.$ Therefore, using \eqref{Hard} in \eqref{lmv}, we obtain
     \begin{align}\label{comv}
         \la_1^{\text{min}}(g_0,V_0)=\frac{\frac{\cns}{2}\uss+\into \uv(x)|u(x)|^2\dx}{\into \ug(x)|u(x)|^2\dx} &\geq \frac{\frac{\cns}{2}[u_H]_s^2+\into \uv^H(x)|u_H(x)|^2\dx}{\into \ug_H(x)|u_H(x)|^2\dx}\\
         &\geq \la_1(\ug_H,\uv^H)\geq \la_1^{\text{min}}(g_0,V_0).
     \end{align}
    As a consequence, we have equality everywhere, in particular, $\la_1^{\text{min}}(g_0,V_0)=\la_1(\ug_H,\uv^H)$ and $u_H$ is an eigenfunction associated to $\la_1(\ug_H,\uv^H).$ Therefore, the following equations hold weakly in $\Om:$
    \begin{align}
        \flap u_H+\uv^H u_H &=\la_1^{\text{min}}(g_0,V_0) \ug_H u_H \;\;\text{in} \;\Omega,\label{eqhh}\\
        \flap u+\uv u&=\la_1^{\text{min}}(g_0,V_0) \ug u \quad\quad\text{in} \;\Omega.\label{eqqq}
    \end{align}
    Let $w(x):=u_H(x)-u(x),$ for $x\in \R^N.$ Then $w\in \xsom$ (as $u,u_H\in \xsom$) and $w$ solves the following equation weakly:
    \begin{equation}\label{difh}
        \flap w+ (\uv^H u_H-\uv u)=\la_1^{\text{min}}(g_0,V_0) (\ug_H u_H-\ug u)\;\text{in}\;\Om\cap H.
    \end{equation}
    We observe that $\uv\geq \uv^H$ in $\Om\cap H.$ Thus we arrive to
    \begin{align}
        \flap w+ \uv w\geq \flap w+ (\uv^H u_H-\uv u) &=\la_1^{\text{min}}(g_0,V_0) (\ug_H u_H-\ug u)\geq 0\;\text{in}\; \Om\cap H,\nonumber\\
        \text{i.e.,}\;\flap w & \geq (-\uv) w\;\text{in}\;\Om\cap H.
    \end{align}
    By the assumption, $V_0\in L^\infty(\Om)$ and as a result $\uv\in L^\infty(\Om).$ Furthermore, as observed in the proof of $(i)$, $w\geq 0$ in $H$ and $w$ is antisymmetric with respect to $\pa H.$ Therefore, appyling the Strong maximum principle (Proposition \ref{Strong_max2}) in $\Om\cap H,$ we conclude either $w\equiv 0$ in $\Om\cap H$ or $\essinf\limits_K\;v>0,$ for all compact $K\subset \Om\cap H,$ i.e.,
    \begin{equation}\label{wh}
        \text{either}\;u_H=u\;\text{in}\;\Om\cap H,\;\text{or}\;u_H>u\;\text{in}\;\Om\cap H.
    \end{equation}
    Now, we analyse both the cases separately.\\
    \underline{\it Case-I}: Suppose $u_H=u$ in $\Om\cap H.$ Thus $u\geq u\circ \si_H$ in $\Om\cap H$. Hence using Remark \ref{Some_remarks}-$(i)$, we have
    $$u_H=u\; \text{in}\;\Om.$$ 
    Next, applying the Proposition \ref{Repre}-$(i)$, we get an increasing function $\phi: \R\rightarrow\R$ such that $\ug=\phi\circ u$ a.e. in $\Om.$ Therefore,
    \begin{equation*}
        \ug\circ \si_H=(\phi \circ u)\circ \si_H\leq \phi \circ u=\ug\;\text{in}\;\Om\cap H.
    \end{equation*}
    i.e., $\ug\circ \si_H\leq \ug$ in $\Om\cap H,$ which gives that $$\ug_H=\ug\; \text{in}\;\Om.$$ Similarly, applying Proposition \ref{Repre}-$(ii)$, we get a decreasing function $\psi:\R\rightarrow\R$ such that $\uv=\psi\circ u$ a.e. in $\Om.$ Now, 
    \begin{equation*}
        \uv\circ \si_H=(\psi \circ u)\circ \si_H\geq \psi\circ u=\uv\;\text{in}\;\Om\cap H.
    \end{equation*}
    i.e., $\uv\circ \si_H\geq \uv$ in $\Om\cap H.$ Thus using Remark \ref{Some_remarks}-$(i),$ we get
    $$\uv^H=\uv\; \text{in}\;\Om.$$
    \underline{\it Case-II}: Let $u_H>u$ in $\Om\cap H.$ This implies $u<u\circ \si_H$ in $\Om\cap H$ and hence (again using Remark \ref{Some_remarks}-$(i)$)
    $$u^H=u\;\text{in}\;\Om.$$
    Proceeding in a similar manner as in the previous case, we get
    $$\ug^H=\ug\;\text{in}\;\Om\;\text{and}\;\uv_H=\uv\;\text{in}\;\Om.$$
    Hence combining both the cases, we conclude that either $u_H=u,\ug_H=\ug,\uv^H=\uv$ in $\Om,$ or $u^H=u,\ug^H=\ug,\uv_H=\uv$ in $\Om.$ This completes the proof.
\end{proof}

Now, we are in a position to prove the Steiner symmetry of the minimizing weights and the corresponding first eigenfunctions of \eqref{Problem} on Steiner symmetric domains.

\begin{proof}[Proof of Theorem \ref{steiner_theo}]
    Let $\ug$ be a minimizing weight and $u$ be a positive first eigenfunction of \eqref{Problem} associated to $\ug.$ Without loss of generality (since fractional Laplacian is invariant under isometries), we can assume that $\Om$ is Steiner symmetric with respect to the hyperplane $\{x_N=0\}$, i.e., we can consider  
    $$H=\{x=(x',x_N)\in \R^N: x_N>0\}.$$
    Recall that, $\H$ is the collection of all open affine half-spaces in $\R^N.$ Now, we define the following class of half-spaces in $\R^N$:
    $$\H_*=\{P\in \H: 0\in \overline{P}, \pa H\subset P,\;\text{and}\;\pa P\;\text{is parallel to}\;\pa H\}.$$
    Since $\si_H(\Om)=\Om$ (Remark \ref{RemS}-$(i)$), we must have $\si_P(\Om)\neq \Om,$ for all $P\in \H_*.$ Furthermore, applying Proposition \ref{Steiner_char}-$(i),$ we get $\Om_P=\Om,$ for all $P\in \H_*.$ Therefore, using Proposition \ref{pola_propo}-$(i),$ we obtain 
$$u_P=u,\;\text{and}\;\ug_P=\ug\;\text{in}\;\Om,\;\text{for all}\;P\in \H_*.$$
    Now, the conclusion follows from Proposition \ref{Steiner_char}-$(ii).$ This finishes the proof.
\end{proof}

Next, as an immediate consequence of the Theorem \ref{steiner_theo}, we prove the radial symmetry of the minimizing weights and the first eigenfunctions in a ball. 
\begin{proof}[Proof of Corollary \ref{rad_theo}]
    Let $\Om=B_r$ be the ball of radius $r$ centered at the origin. Recall that (Section \ref{fol_sub}),
    $$\widehat{\H}_0=\{H\in \H: 0\in \pa H\}.$$
    It is easy to observe that $\Om$ is Steiner symmetric with respect to $\pa H,$ for each $H\in \widehat{\H}_0$ (Remark \ref{RemS}). Therefore, Theorem \ref{steiner_theo} infers that both $u$ and $\ug$ are Steiner symmetric with respect to $\pa H,$ for each $H\in \widehat{\H}_0.$ Consequently, $u$ and $\ug$ are symmetric with respect to the hyperplane $\pa H,$ for any $H\in \widehat{\H}_0.$ Hence, the radiality follows.  
\end{proof}

\begin{remark}
In the local case, various symmetry, such as, radial symmetry, Steiner symmetry, etc., of the minimizing weights and the eigenfunctions have been obtained in \cite{CuccuF, Prajapat, Chanillo, Anedda} under different assumptions on $g_0.$ The proofs of these results make use of either Schwarz or Steiner symmetrization and a rigidity results regarding the equality case of the well-known P\'{o}lya-Szeg\"{o} inequality (see \cite[Theorem 1.1]{Brothers} for Schwarz symmetrization and \cite[Theorem 2.2]{Cianchi} for Steiner symmetrization). However, it is worth mentioning that our approach of proofs is quite different from those available in the referred articles.   
\end{remark}

In the following, we prove the foliated Schwarz symmetry of the minimizing weights and the first eigenfunctions on annular domains.

\begin{proof}[Proof of Theorem \ref{foliated_theo}]
    Let $\B_t$ be the annular domain  defined by \eqref{annulus}, i.e., $$\B_t=B_R(0)\setminus \overline{B_r(te_1)}, \;\text{where}\;0\leq t<R-r.$$ 
    Also, let $\ug,\uv$ be the minimizing weights in \eqref{min_prob}.
    
    $(i)$ \underline{\textit{Concentric case}}: Let $t=0.$ Also, let  $\widehat{\H}_0$ be as mentioned in the proof of Corollary \ref{rad_theo}. For $H\in \widehat{\H}_0,$ we have $\si_H(\B_0)=\B_0$ and $(\B_0)_H=\B_0.$ Thus applying Proposition \ref{pola_propo}-$(ii)$, we get 
$$\text{either}\;u_H=u\;\text{or}\;u^H=u\;\text{in}\;\B_0.$$
Therefore, we use Proposition \ref{Foliated_char}-$(i)$ to get $\ga\in \S^{N-1}$ such that $u$ is foliated Schwarz symmetric in $\B_0$ with respect to $\ga.$ Hence from Proposition \ref{Foliated_char}-$(ii),$ we must have 
$$u_H=u\;\text{in}\;\B_0,\;\text{for all}\;H\in \widehat{\H}_0(\ga).$$
Now, following the same method as in the proof (\textit{Case-I}) of Proposition \ref{pola_propo}-$(ii)$ and using the characterization of $\ug$ and $\uv$ in terms of $u$ (Proposition \ref{Repre}), we conclude that
\begin{equation}\label{ee}
\ug_H=\ug\;\text{and}\;\uv^H=\uv\;\text{in}\;\B_0,\;\text{for all}\;H\in \widehat{\H}_0(\ga),
\end{equation}
where $\widehat{\H}_0(\ga)=\{H\in\widehat{\H}_0 :\ga\in H\}.$ Thus Proposition \ref{Foliated_char}-$(ii)$ ensures that $\ug$ is foliated Schwarz symmetric in $\B_0$ with respect to $\ga.$ Next, from the definition of dual-polarization (Definition \ref{pola_def}-$(iii)$), it follows that $\uv^H=\uv_{\overline{H}^c},$ for any half-space $H\subset \R^N.$ Therefore, \eqref{ee} yields 
$\uv_{\overline{H}^c}=\uv\;\text{in}\;\B_0,\;\text{for all}\;H\in \widehat{\H}_0(\ga),$ and it is equivalent to the following: 
$$\uv_H=\uv\;\text{in}\;\B_0,\;\text{for all}\;H\in \widehat{\H}_0(-\ga).$$
Therefore, $\uv$ is foliated Schwarz symmetric in $\B_0$ with respect to $-\ga$ (by Proposition \ref{Foliated_char}-$(ii)$).

$(ii)$ \underline{\textit{Non-concentric case}}: Let $t>0$. By the hypothesis, $V_0=0.$
Observe that, for $t>0,$ we have $\si_H(\B_t)\neq \B_t$ and $(\B_t)_H=\B_t,$ for all $H\in \widehat{\H}_0(-e_1).$ Therefore, we apply Proposition \ref{pola_propo}-$(i)$ to get 
$$u_H=u\;\text{and}\;\ug_H=\ug\;\text{in}\;\B_t,\;\text{for all}\;H\in \widehat{\H}_0(-e_1).$$
Now, using the characterization of the foliated Schwarz symmetry (Proposition \ref{Foliated_char}-$(ii)$), we conclude that $u$ and $\ug$ is foliated Schwarz symmetric in $\B_t$ with respect to $-e_1$. The proof is completed.
\end{proof}

\begin{corollary}\label{compo}
    Let $\B_0$ be as in Theorem \ref{foliated_theo}, $g_0=1$ and $V_0=\al\chi_D,$ for some $\al>0$ and $D\subset \B_0.$ Then $\uv=\al \chi_{\underline{D}}$, for some $\underline{D}\subset \B_0$ with $|\underline{D}|=|D|$. Furthermore, $\underline{D}$ is symmetric with respect to the axis $\ga\R,$ for some $\ga\in \S^{N-1}.$ 
\end{corollary}
\begin{proof}
   By assumption, $V_0=\al\chi_D,$ for some $\al>0$ and $D\subset \B_0.$ Now using \eqref{Rearrange_def}, it is easy to observe that if $V\in \E(V_0),$ then $V=\al\chi_A,$ for some $A\subset \B_0$ with $|A|=|D|.$ Since $\uv\in \E(V_0),$ we obtain $\uv=\al \chi_{\underline{D}}$, for some $\underline{D}\subset \B_0$ with $|\underline{D}|=|D|$. Moreover, using Theorem \ref{foliated_theo}-$(i),$ we get $\ga\in\S^{N-1}$ such that $\uv$ is foliated Schwarz symmetric in $\B_0$ with respect to $-\ga.$  As a consequence, $\uv$ and equivalently, the set $\underline{D}$ is symmetric with respect to the axis $\ga \R.$ 
\end{proof}

\begin{remark}\label{asym}
    As indicated earlier, the minimizing weights in \eqref{min_prob} and the corresponding eigenfunction of \eqref{Problem} may not inherit the full symmetry of the underlying domain in general. More explicitly, in \cite[Theorem 1.4]{LeeLee}, for $g_0=1$ and $V_0=\al\chi_D$ $(\al>0\;\text{and} \;D\subset \B_0),$ the authors constructed an annular domain $\B_0=B_{R+1}(0)\setminus \overline{B_R(0)},$ for sufficiently large $R,$ where the minimizing weight $\uv=\al \chi_{\underline{D}}$ (for some $\underline{D}\subset \B_0$ with $|\underline{D}|=|D|$) does not satisfy the rotational symmetry as $\B_0$. Nevertheless, Corollary \ref{compo} implies that the optimal set $\underline{D}$ is symmetric with respect to the axis $\ga \R.$ Moreover, it is necessary to mention that such axial symmetry holds irrespective of the size of $R,$ which is indeed a novelty of Theorem \ref{foliated_theo}.     
\end{remark}


\noi\textbf{Acknowledgements.} The author is grateful for the support provided by TIFR Centre for Applicable Mathematics.

\bibliographystyle{abbrvurl}

\bibliography{Reference}
\end{document}